\documentclass[a4paper,11pt]{article}
\usepackage[utf8]{inputenc}

\usepackage{verbatim}
\usepackage{amsfonts}
\usepackage{a4wide}
\usepackage{amsthm}
\usepackage{color}
\usepackage{graphicx}
\usepackage{amsmath,amssymb}
\usepackage{lipsum}
\usepackage{fancyhdr}
\usepackage{hyperref} 
\usepackage[numbers]{natbib} 
\usepackage{doi}

\newcommand{\bv}{\boldsymbol v}
\newcommand{\bu}{\boldsymbol u}
\newcommand{\bff}{\boldsymbol f}
\newcommand{\bw}{\boldsymbol w}

\newcommand{\bz}{\boldsymbol z}

\renewcommand{\div}{\mathrm{div}\,}

\newcommand{\supp}{\mathrm{supp}\,}

\newcommand{\eps}{\varepsilon}

\newcommand{\R}{\mathbb R}

\newcommand{\Rn}{\mathbb R^n}
\newcommand{\Rm}{\mathbb R^m}
\DeclareMathOperator*{\osc}{osc}

\newcommand{\res}{\,\raisebox{-.127ex}{\reflectbox{\rotatebox[origin=br]{-90}{$\lnot$}}}\,}

\newtheorem*{thm}{Theorem}
\newtheorem*{lemma}{Lemma}

\newtheorem*{cor}{Corollary}
\theoremstyle{remark}

\newcommand{\ignore}[1]{{}}
\author{Lorenzo Giacomelli$^*$, Micha{\l} {\L}asica$^{*\dagger}$ \\
}

\title{A local estimate for vectorial total variation minimization\\ in one dimension}

\date{\today}

\begin{document}
 \maketitle

 \begin{abstract}
 	Let $\bu$ be the minimizer of vectorial total variation ($VTV$) with $L^2$ data-fidelity term on an interval $I$. We show that the total variation of $\bu$ over any subinterval of $I$ is bounded by that of the datum over the same subinterval. We deduce analogous statement for the vectorial total variation flow on $I$.
\end{abstract}

{\small
}

\thispagestyle{fancy}
\fancyhf{}
\lfoot{\small $^*$ SBAI Department, Sapienza University of Rome \\
	$^\dagger$ Corresponding author.\\ Institute of Applied Mathematics and Mechanics, University of Warsaw\\ 
ul.\;Banacha 2, 02-097 Warszawa, Poland \\
\tt{lasica@mimuw.edu.pl}} 
\section{Introduction}
\noindent
Let $\Omega$ be a bounded domain in $\Rm$, $m\ge 1$. The (vectorial) total variation functional
$TV \colon L^2(\Omega, \Rn) \to [0, +\infty]$, $n\geq1$ is defined by
\begin{equation} \label{deftv}
TV(\bu) = \sup\left\{\int_\Omega \bu \cdot \div \boldsymbol \varphi \colon\boldsymbol \varphi \in C^1_c(\Omega, \mathbb R^{m \times n}), |\boldsymbol \varphi| \leq 1\right\} = \left\{ \begin{array}{ll} |D \bu|(\Omega) & \text{if }\bu \in BV(\Omega, \Rn), \\ + \infty & \text{otherwise.}\end{array}\right.
\end{equation} 
Here and in the following, if $\bv$ is a vector in a Euclidean space (e.\,g.\;$\mathbb R^{m \times n}$), $|\bv|$ denotes its Euclidean norm, while if $\bw$ is a vector measure on $\Omega$, $|\bw|$ denotes its variation with respect to the Euclidean norm.

Given $\bff\in L^2(\Omega, \Rn)$ and $\lambda>0$, the $VTV$-$L^2$ denoising model \cite{rof,chambollelions,snakes,blomgrenchan} amounts to solving the following minimization problem:
\begin{equation} \label{l2tv}
\min_{\bw\in L^2(\Omega, \Rn)}E(\bw), \qquad E(\bw):= TV(\bw) + \tfrac1{2\lambda} \int_\Omega |\bw-\bff|^2.
\end{equation}
By strict convexity of $E$, \eqref{l2tv} has exactly one solution $\bu\in BV(\Omega, \Rn)$. It is clear that 
\begin{equation}\label{energyest}
TV(\bu) \leq TV(\bff).
\end{equation}
Here we consider the setting where $m=1$ and $\Omega = I = ]a,b[$ is an interval, focusing on vectorial case $n\geq 2$. This corresponds to denoising one-dimensional multichannel data. A typical source of such data is tracking orientation of objects (cameras, proteins, aircraft, etc.) in time (see \cite{lellmann, weinmanndemaretstorath} and references in \cite{weinmanndemaretstorath}). The orientation data are constrained to a Riemannian manifold, which however can be locally approximated by a Euclidean space via the Riemannian logarithmic map (see also discussion at the end of this section). It turns out that in the one-dimensional setting, estimate \eqref{energyest} can be localized in an unusually robust way.
\begin{thm} \label{thm-denoising}
Let $\bff\in L^2(I, \Rn)$. Suppose that $\bff \in BV(U, \Rn)$ for an open interval $U\subset I$. Then the minimizer $u$ of $E$ satisfies $|\bu'| \leq |\bff'|$ in the sense of Borel measures on $U$, i.\,e.
	\begin{equation} \label{theest}
|\bu'|(A) \leq |\bff'|(A) \quad \text{for any Borel } A \subset U.
	\end{equation}
\end{thm}
Let us note a simple, but already interesting fact that follows trivially from the Theorem: for any $\bff \in L^2(I, \mathbb R^n)$ and any open interval $U \subset I$ there holds 
\[TV_U(\bu)\leq TV_U(\bff),\]
where $TV_U$ is defined as in \eqref{deftv}, with $\Omega$ changed to $U$. 

Another immediate consequence of the Theorem is that if $\bff$ belongs to any subspace of $BV(I, \Rn)$ defined in terms of a bound on $|\bff'|$, such as $W^{1,p}(I, \Rn)$, $p\in [1,\infty]$ or $SBV(I, \Rn)$, then $\bu$ belongs to the same subspace.

A closely related, natural way of diminishing $TV$ is to follow its $L^2$-steepest descent flow
\begin{equation}\label{tvflow}
\begin{aligned}  S\colon &[0,\infty[ \times L^2(\Omega, \Rn) \ni (t, \bu_0)\mapsto \bu(t, \cdot) \in L^2(\Omega, \Rn), \\
 &\tfrac{\partial \bu}{\partial t}(t,\cdot) \in - \partial TV(\bu(t,\cdot)) \text{ for } t>0, \quad \bu(0, \cdot) = \bu_0. 
 \end{aligned}
\end{equation}
This is the so-called \emph{total variation flow}. It can be computed via the exponential formula
\begin{equation} \label{expform}
	\bu(t, \cdot) = S(t,\bu_0) = \lim_{n \to \infty} \left(R_{t/n}\right)^{n} (\bu_0)
	\end{equation}
 for $t>0$, $\bu_0 \in L^2(\Omega, \Rn)$, where $R_{t/n}(\boldsymbol w)$ is the minimizer of $E$ with $\lambda = t/n$, $\bff = \boldsymbol w$ \cite[Corollary 4.4]{operateurs}.
The total variation flow has been
extensively investigated by mathematicians in parallel to problem \eqref{l2tv} (see \cite{acmbook} and other references cited below). An immediate consequence of the Theorem, \eqref{expform} and lower semicontinuity of $TV$ is

\begin{cor}
Let $\bu_0 \in L^2(I, \Rn)$ and let $\bu$ be as in (\ref{tvflow}). Suppose that $\bu_0 \in BV(U, \Rn)$ for an open interval $U\subset I$. Then
$|\frac{\partial\bu}{\partial x}(t,\cdot)| \leq |\frac{\partial \bu_0}{\partial x}|$ in the sense of Borel measures on $U$, i.\,e.
	\begin{equation} \label{theestf}
	\left|\tfrac{\partial\bu}{\partial x}(t,\cdot)\right|(A) \leq \left|\tfrac{\partial\bu_0}{\partial x}\right|(A) \quad \text{for any Borel }A \subset U.
	\end{equation}
\end{cor}

Inequality (\ref{theest}) (resp.\;(\ref{theestf})) is an example of a  \emph{completely local} estimate for a solution to a variational problem (evolutionary equation). By this we mean that the value of a \emph{local} functional on the solution space (in our case $L^2(I,\Rn)$) evaluated on the solution is estimated by the value of the same functional on the datum. Several results like that, which we now briefly discuss, are already known for steepest descent flows of \emph{scalar} $TV$ (i.\,e.\;for $n=1$) and similar functionals. Usually the authors first obtain the estimates for the resolvent variational problem \eqref{l2tv} or its analog. For consistency, we discuss them in the language of the flow.

In \cite{bonfortefigalli}, the authors consider scalar total variation flow on an interval $I$. They analyze the evolution of step functions (which is a class preserved by the flow) and use $L^q$-contractivity of the flow. In this way, they prove that if $u$ is the solution starting with initial datum $u_0 \in BV(I)$, the size of jumps of $u(t,\cdot)$, $t>0$, is not bigger than the size of corresponding jumps of $u_0$ and that
\begin{equation}\label{osc}\osc_{J} u(t,\cdot) \leq \osc_{J} u_0 \quad \text{for }t>0
\end{equation}
on any open interval $J\subset I $ over which $u_0$ is continuous (recall that $\osc_J v = \sup_J v - \inf_J v$ for $v \in C(J)$). The authors remark that this implies preservation of $W^{1,1}(I)$ and $C^{0,\alpha}(I)$, $\alpha \in ]0,1]$ regularity by the flow. We note that preservation of $C^{0,\alpha}(\Omega)$ class is also known for the scalar total variation flow on a convex domain $\Omega \subset \mathbb R^m$ \cite{CaChNoreg, mercier}. On the other hand, preservation of $W^{1,1}(I)$ regularity has recently been obtained for gradient flows of more general functionals of linear growth at infinity on an interval \cite{nakayasurybkaw11}.

In \cite{bcng}, the authors consider the gradient flow of a functional $\bu \mapsto \int_\Omega |\div \bu|$, where $\Omega$ is a bounded domain in $\Rm$. For a solution $\bu$ starting with $\bu_0 \in L^2(\Omega, \Rm)$ such that $\div \bu_0$ is a Radon measure on $\Omega$, they prove that
\[(\div \bu(t, \cdot))_\pm \leq (\div \bu_0)_\pm \quad \text{for }t>0\]
in the sense of measures. This coincides with our result in the case $m=n=1$. Their technique is based on considering the dual problem to the variational semi-discretization of the flow and involves a comparison principle. In the essential lemma, they show certain monotonicity property of level sets of the solution to the dual problem with respect to the parameter of discretization. This does not seem to be adaptable to the case where the divergence is a vectorial quantity (which would cover our result for $m=1$). On the other hand, at least at the formal level our method can be adapted to the case of gradient flow of $\int_\Omega |\div \bu|$ with $\bu\colon \Rm\supset \Omega\to \R^{m\times n}$.
%
%our technique could be transferred to the case of gradient flow of $\int_\Omega |\div \bu|$. We do not explicitly include this case in order not to obfuscate the idea of proof.
%

We mention that in the case of isotropic {\em scalar} total variation flow on an $m$-dimensional domain $\Omega$, a completely local estimate of type \eqref{theestf} is available \cite{jalalzaijump} for the jump part of the gradient of the solution $u$ starting with $u_0 \in BV(\Omega) \cap L^m(\Omega)$:
\begin{equation} \label{jumpest}
|D u(t, \cdot)|\res J_{u(t,\cdot)} \leq |D u_0|\res J_{u_0}
\end{equation}
as Borel measures. Our main result contains in particular the first extension of \eqref{jumpest} to the vectorial case, $n>1$. It is also well known that such an estimate does not hold in general for the absolutely continuous part of $|D u|$. For a counterexample, one can take as $u_0$ the characteristic function of a \emph{non-calibrable} convex set, such as a square in the plane \cite{altercc}. To our knowledge, it is an open question whether an estimate analogous to \eqref{jumpest} holds for the remaining Cantor part.

On the other hand, for the orthotropic scalar total variation flow, local estimate of form \eqref{jumpest} may not hold, as the jump set of solution may expand compared to the initial datum \cite[Example 4]{moll}. In fact, jumps can appear even if they are absent in the initial datum \cite[Example 3]{tetris}.

Our technique is based on local integral estimates and a careful choice of the approximating scheme (see Section 2). For this reason, the Theorem and the Corollary can be transferred to different (e.\,g.\;Dirichlet) boundary conditions. We also expect that a similar technique can be used to prove analogous versions of the results when the range of $\bu$ is constrained to a Riemannian submanifold in $\mathbb R^n$. In fact, the existence theory for constrained total variation flows (also known as \emph{$1$-harmonic flows}) is at present limited to non-generic targets \cite{giacomellimazonmoll2013,giacomellimazonmoll2014,dicastrogiacomelli} (not including the orientation space $SO(3)$) and Lipschitz initial data \cite{gigakashimayamazaki, giacomellilasicamoll}. We believe that an a priori estimate of form \eqref{theestf} will provide a convenient tool for generalizing the existence theory to the case of initial data of bounded variation into generic target manifolds.

\section{The proof}

The proof of the Theorem is based on a uniform estimate for a family of regularizations of $E$:
\begin{equation}
\label{def-Eeps}
E_\eps(\bu):= \int_I \left(\tfrac1{2\lambda}|\bu-\bff|^2 +|\bu'|_\eps +\tfrac{\eps^2}{2} |\bu'|^2\right), \qquad |\bv|_\eps:= \sqrt{\bv^2+\eps^2}.
\end{equation}
For any $\bff\in L^2(I)$, $E_\eps$ has a unique minimizer $\bu_\eps \in H^1(I)$ (at this point, we stop specifying the codomain in the notation for function spaces). As $E_\eps$ is smooth and uniformly convex for a fixed $\eps$, a standard variational argument yields that $\bu_\eps\in H^2(I)$ (hence also $\bu_\eps \in C^1(\overline I)$) and $\bu_\eps$ is a strong solution to the Euler-Lagrange system
\begin{equation}\label{ELeps}
\tfrac{1}{\lambda} (\bu_\eps-\bff) = \left(\frac{\bu_\eps'}{|\bu_\eps'|_\eps}\right)' +\eps^2\bu_\eps'' \ \text{ in } I, \qquad \bu_\eps'=0\ \text{ on } \partial I.
\end{equation}

For $x_0 \in \mathbb R$ and $r>0$, let us denote by $B_r(x_0)$ the \emph{closed} interval $[x_0-r,x_0+r]$.

 \begin{lemma} %\label{estimateorder}
Let $\bff\in H^{1}(I)$ and let $\bu_\eps$ be the unique minimizer of $E_\eps$. Let $x_0 \in I$ and $0<r<R$ such that $B_R(x_0) \subset I$. Then
\begin{equation}\label{global-j}
\int_I|\bu_\eps'|_\eps^{p} \le p \eps^p |I| + \int_I |\bff'|^p \quad\text{for all } p\in [1,2]
\end{equation}
and
\begin{equation} \label{approxbound}
\limsup_{\eps\to 0} \int_{B_r(x_0)}|\bu_\eps'|_\eps^{p} \le  \int_{B_R(x_0)} |\bff'|^p\quad\text{for all } p\in ]1,2].
\end{equation}
  \end{lemma}
 \begin{proof}
In the following calculations we will omit the index $\eps$. Let $\varphi$ be a Lipschitz function on $I$. Given $p \geq 1$, we multiply \eqref{ELeps} by $-\left(\varphi^2 |\bu'|_\eps^{p-1}\bz\right)'$, where $\bz = \frac{\phantom{|}\bu'\phantom{|_\eps}}{|\bu'|_\eps}$,
obtaining (after one integration by parts on the l.\,h.\,s.)
\begin{multline}
\label{pnorm}
J:=\tfrac{1}{\lambda}\int_I \varphi^2 |\bu'|_\eps^{p-1}\bz\cdot (\bu'-\bff') =  -2\int_I \varphi\, \varphi' |\bu'|_\eps^{p-1} \bz \cdot (\bz'+\eps^2\bu'') \\ - (p-1) \int_I \varphi^2 |\bu'|_\eps^{p-3}\bu' \cdot \bu'' \, \bz \cdot (\bz'+\eps^2\bu'') - \int_I \underbrace{\varphi^2 |\bu'|_\eps^{p-1} \bz'\cdot (\bz'+\eps^2\bu'')}_{\ge 0}.
\end{multline}
We have
\begin{equation}\label{zzx}
\bz \cdot \bz' = \tfrac{1}{2} \left(|\bz|^2\right)' = \tfrac{1}{2}\left(\frac{|\bu'|^2}{|\bu'|^2 + \eps^2}\right)' \\
= \tfrac{1}{2} \left( 1 - \frac{\eps^2}{|\bu'|^2 + \eps^2}\right)' = %\frac{\eps^2 \bu_x \cdot \bu_{xx}}{\left(|\bu_x|^2 + \eps^2\right)^2}=
\frac{\eps^2 \bu' \cdot \bu''}{|\bu'|_\eps^4}.
\end{equation}
%and
%\begin{equation}\label{zzx}
%\eps^2 \bz \cdot \bu_{xx} = \frac{\eps^2 \bu_x \cdot \bu_{xx}}{|\bu_x|_\eps}.
%\end{equation}
%and
%\begin{equation}\label{simple}
%\bz\cdot \bu_x= \frac{|\bu_x|^2}{|\bu_x|_\eps}.
%\end{equation}
Substituting \eqref{zzx} into \eqref{pnorm} and rearranging terms yields
\begin{multline} \label{pnorm2}
\tfrac{1}{\eps^2} J \leq -  2 \int_I \varphi\, \varphi' |\bu'|_\eps^{p-5} (1+|\bu'|_\eps^{3})\bu' \cdot \bu''- (p-1) \int_I \varphi^2 |\bu'|_\eps^{p-7} (1+|\bu'|_\eps^{3})\left| \bu' \cdot \bu''\right|^2.
\end{multline}
For $\varphi\equiv 1$, this implies that
$$
\int_I |\bu'|_\eps^{p-2}|\bu'|^2 \le \int_I |\bu'|_\eps^{p-2}\bu'\cdot \bff' \le \tfrac{p-1}{p} \int_I |\bu'|_\eps^{p} +\tfrac1p\int_I |\bff'|^p,
$$
and since
$$
|\bu'|_\eps^{p-2}|\bu'|^2=  |\bu'|_\eps^{p} -\eps^2|\bu'|_\eps^{p-2}
\stackrel{p\le 2} \ge|\bu'|_\eps^{p} -\eps^p
$$
we obtain \eqref{global-j}.

Let now $\varphi$ be defined by $\supp\varphi = B_R(x_0)$, $\varphi = 1$ in $B_r(x_0)$ and $|\varphi'| =
\frac{1}{R-r}$ in $B_R(x_0)\setminus B_r(x_0)$. Applying Cauchy-Schwarz inequality, for $p\in ]1,2]$ we obtain from \eqref{pnorm2} and \eqref{global-j} that
\begin{equation} \label{pnorm3}
 J \leq  \tfrac{\eps^2}{p-1} \int_I |\varphi'|^2 (|\bu'|_\eps^{p-3}+|\bu'|_\eps^{p})  \le \tfrac{2\eps^{p-1}}{(p-1)(R-r)} + \tfrac{\eps^2}{(p-1)(R-r)^2} \left(p \eps^p |I| + \int_I |\bff'|^p\right)
\end{equation}
On the other hand, arguing as in the proof of  \eqref{global-j} we have
\begin{equation}\label{conc}
J\ge \int_{B_r(x_0)}|\bu'|_\eps^{p} - p \eps^p |R| - \int_{B_R(x_0)} |\bff'|^p.
\end{equation}
Combining \eqref{pnorm3} and \eqref{conc} we obtain
$$
\int_{B_r(x_0)}|\bu'|_\eps^{p} \le  \int_{B_R(x_0)} |\bff'|^p + p \eps^p |R| + \tfrac{2\eps^{p-1}}{(p-1)(R-r)} + \tfrac{\eps^2}{(p-1)(R-r)^2} \left(p \eps^p |I| + \int_I |\bff'|^p\right),
$$
whence \eqref{approxbound}.
\end{proof}
%Let us remark that expanding the first term on the right hand side of \eqref{pnorm} shows it to be of order $\eps^4$, hence it is not useful for obtaining a completely local estimate. For this reason, approximation step with $p>1$ cannot be omitted in presented reasoning.

With the Lemma at hand, it is easy to conclude.

\begin{proof}[Proof of the Theorem]

We first consider $\bff\in H^1(I)$. In view of \eqref{global-j}, up to a subsequence $\bu_\eps$ converges weakly in $H^1(I)$ and uniformly in $C(\overline I)$ to the unique minimizer $\bu \in H^1(I)$ of $E$ ($\Gamma$-convergence of $E_\eps$ to $E$ is straightforward). From \eqref{approxbound} we see that
\[
\int_{B_r(x_0)}|\bu'|^{p} \le  \int_{B_R(x_0)} |\bff'|^p \quad\text{for all } p\in ]1,2]
\]
whence, after limit passages $p \to 1^+$, $R\to r^+$,
    \begin{equation}\label{theboundk}
    \int_{B_r(x_0)} | \bu'| \leq \int_{B_r(x_0)} |\bff'|.
    \end{equation}
Finally, we remove the smoothness assumption on $\bff$. Modifying slightly the standard mollification technique, we obtain $(\bff_k) \subset C^\infty(\overline I)$ converging to $\bff$ strongly in $L^2(I)$ and strictly in $BV(U)$ as $k \to \infty$. Let $\bu_k$ denote the minimizer of
\[
E_k(\bu):=\tfrac1{2\lambda} \int_I |\bu-\bff_k|^2 + TV(\bu).
\]
As
\[ TV(\bu_k) \leq E_k(\bu_k) \leq E_k(0) = \tfrac1{2\lambda} \int_I |\bff_k|^2,\]
the sequence $\bu_k$ is uniformly bounded in $BV(I)$. We extract (and relabel) a subsequence such that
       \begin{equation*}
   \bu_k \stackrel{*}{\rightharpoonup} \bu \quad \text{in } BV(I),\qquad \bu_k \to \bu \quad \text{in } L^2(I)
   \end{equation*}
Since $\bff_k\to \bff\in L^2(I)$, $E_k$ $\Gamma$-converges to $E$, hence $\bu_k$ converges to its unique minimizer. By lower semicontinuity of total variation and strict convergence of $\bff_k$ on $U$, passing to the limit $k\to \infty$ in \eqref{theboundk} yields
   \begin{equation}\label{fgh}
   | \bu'|(B_r(x_0)) \leq |\bff'|(B_r(x_0))
   \end{equation}
   for any ball $B_r(x_0) \subset U$ such that (cf.\;\cite[Proposition 3.7]{afp})
   \begin{equation} \label{goodballs}
   |\bff'|(\partial B_r(x_0)) = 0.
   \end{equation}
Property \eqref{goodballs} is satisfied for every $x_0 \in U$ and almost every $r>0$ such that $B_r(x_0) \subset U$. Hence, by \cite[1.5.2., Corollary 1]{evansgariepy}, up to a set of zero $|\bu'|$ measure we can fill any open set $V \subset U$ with a countable family of disjoint closed balls $B_{r_j}(x_j)$ contained in $V$ and satisfying \eqref{goodballs}: hence
\[|\bu'|(V) = \sum_{j=1}^\infty |\bu'|\left(B_{r_j}(x_j)\right) \leq \sum_{j=1}^\infty |\bff'|\left(B_{r_j}(x_j)\right) \leq |\bff'|(V). \]
Finally, by virtue of \cite[1.1.1., Lemma 1]{evansgariepy}, given a Borel $A \subset I$ and $\delta>0$ we can find an open $V$ with $A \subset V$ and $|\bff'|(V\setminus A) \leq \delta$. Therefore,
   \[|\bu'|(A)\leq |\bu'|(V) \leq |\bff'|(V) \leq |\bff'|(A) +\delta.
\]
As $\delta>0$ is arbitrary, we are done.
\end{proof}

 \bibliographystyle{twojstyl}
%\footnotesize
 \bibliography{aniso}
\end{document}